\documentclass[11pt, a4paper]{amsart}
\usepackage{amssymb,amsmath,amsthm,amsfonts, a4wide}
\usepackage{amsrefs}
\usepackage[mathscr]{euscript}

\renewcommand{\mathcal}{\mathscr}

\DeclareFontFamily{OT1}{pzc}{}
\DeclareFontShape{OT1}{pzc}{m}{it}{<-> s * [1.10] pzcmi7t}{}
\DeclareMathAlphabet{\mathpzc}{OT1}{pzc}{m}{it}

\usepackage[cal=boondox]{mathalfa}

\def\R {\mathbb{R}}

\def\Z {\mathbb{Z}}
\def\div{\mathrm{div}}
\def\eps{\varepsilon}
\newcommand{\pcors}{\mathcal{p}}
\newcommand{\qcors}{\mathcal{q}}

\usepackage{graphicx}
\usepackage{amssymb}
\usepackage{epstopdf}
\newtheorem{proposition}{Proposition}[section]
\newtheorem{theorem}[proposition]{Theorem}
\newtheorem*{theorem*}{Theorem}

\newtheorem{lemma}[proposition]{Lemma}

\theoremstyle{definition}

\newtheorem{remark}[proposition]{Remark}
\numberwithin{equation}{section}

\title{On stationary fractional mean field games}
\thanks{This work has been supported by the
Andrew Sisson Fund 2017, 
INdAM Intensive Period
``Contemporary Research in elliptic PDEs and related topics'',
the Fondazione CaRiPaRo
Project ``Nonlinear Partial Differential Equations:
Asymptotic Problems and Mean-Field Games", Project PRA 2017 of the University of Pisa "Problemi di ottimizzazione e di evoluzione in ambito variazionale",
 the INdAM-GNAMPA projects 
``Problemi nonlocal e degeneri nello spazio Euclideo" and ``Tecniche EDP, dinamiche e probabilistiche per lo studio di problemi asintotici".}

\author[A. Cesaroni]{Annalisa Cesaroni}
\author[M. Cirant]{Marco Cirant}
\author[S. Dipierro]{Serena Dipierro}
\author[M. Novaga]{Matteo Novaga}
\author[E. Valdinoci]{Enrico Valdinoci}

\address{
{\em Annalisa Cesaroni}:
Dipartimento di Scienze Statistiche,
Universit\`a di Padova, Via Battisti 241/243, 35121 Padova, Italy. {\tt annalisa.cesaroni@unipd.it} }
\address{
{\em Marco Cirant}:
Dipartimento di Matematica Tullio Levi Civita,
Universit\`a di Padova, Via Trieste 63, 35121 Padova, Italy. {\tt cirant@math.unipd.it} }
\address{
{\em Serena Dipierro}:
Dipartimento di Matematica, Universit\`a di Milano,
Via Saldini 50, 20133 Milan, Italy, and
School of Mathematics
and Statistics,
University of Melbourne, 813 Swanston St, Parkville VIC 3010, Australia. 
{\tt sdipierro@unimelb.edu.au}}
\address{{\em Matteo Novaga}: Dipartimento di Matematica,
Universit\`a di Pisa, 
Largo Pontecorvo 5, 56127 Pisa,
Italy. {\tt matteo.novaga@unipi.it}}
\address{{\em Enrico Valdinoci}:
School of Mathematics
and Statistics,
University of Melbourne, 813 Swanston St, Parkville VIC 3010, Australia,
Dipartimento di Matematica, Universit\`a di Milano,
Via Saldini 50, 20133 Milan, Italy, and IMATI-CNR, Via Ferrata 1, 27100 Pavia,
Italy. {\tt enrico.v@unimelb.edu.au}}

\date{} 

\begin{document}

\begin{abstract}
We provide  an existence result for  stationary fractional mean field game systems, with fractional exponent greater than ${1/2}$. 
In the case in which the coupling is a nonlocal regularizing potential, we obtain  existence of solutions under general assumptions on the Hamiltonian. 
In the case of local coupling, we restrict to the subcritical regime, that is the case in which the diffusion part of the operator dominates the Hamiltonian term. 
We consider both the case of local bounded coupling and of local unbounded coupling with power-type growth. In this second regime, 
we impose some conditions on the growth of the coupling and on the growth of the Hamiltonian with respect to the gradient term. 

\bigskip
 
\noindent
{\footnotesize \textbf{AMS-Subject Classification}}. {\footnotesize 35R11, 49N70, 35J47, 35Q84, 91A13.}\medskip


\end{abstract}

%35R11 Fractional PDE
%49N70 Differential games
%35J47 Elliptic systems
%35Q84 Fokker Planck equations
%91A13 Games with infinitely many players
 \maketitle

\tableofcontents

\section{Introduction}\label{sec:intro}
Mean Field Games (briefly MFG) is a very recent mathematical theory
modelling the macroscopic behaviour of a large population of indistinguishable agents
who wish to minimize a cost depending on the distribution of all the agents  in a noisy environment.
It was proposed independently in~2006 by Lasry and Lions (\cite{ll})
and Huang, Caines and Malham\`e (\cite{hcm}), and it has a number of potential applications,
from economics and finance (growth theory, environmental policy,
formation of volatility in financial markets), to engineering and models of social systems,
such as crowd motion and traffic control. For the development of the theory
and several applications, we refer to the monographs~\cite{bf},
\cite{gomesbook}, and to the references therein. 

Up to now, the noisy environment in which the average game takes place has been usually
modeled  by standard diffusion.
Our aim is to consider a more general framework for the disturbances, and
in particular we take into account processes driven by  pure jump {L}\'evy processes.
This generalization is interesting for applications to financial models, where
jump processes are widely used to model sudden crisis and crashes on the markets (see e.g. the monograph \cite{rama} 
for  a detailed description of  motivations for 
the use of processes with jumps in financial models and examples of applications of {L}\'evy processes in risk management). 

More precisely, even if still heuristically, MFG are noncooperative differential games, with a continuum of players,  each of whom
controls his own  trajectory in the state space, which in our case is the $N$-dimensional torus. 
The trajectory of each player is affected by a  fractional Brownian motion: it is defined by a  stochastic differential equation
\[dX_t=v_t dt+ dZ_t,\] 
where $v_t$ is the control and  $Z_t$ is a $N$-dimensional, $2s$-stable pure jumps  {L}\'evy process, with associated {L}\'evy measure 
(which describes the distribution of jumps of the process) given by 
$\nu(dx)=\frac{1}{|x|^{N+2s}}dx$ (see~\cite{a}).
Each player wants to minimize the long time average cost
\[\liminf_{T\to+\infty}\, \mathbf{E} \left[ \frac{1}{T}\, \int_0^T L(v_s) +f(X_s, m(X_s))ds\right],\]
where $m(x)$ denotes the density distribution of the  population at point $x$, $L(\qcors)$ is a superlinear convex  
function and $f$ is a cost function taking into account 
the position of each player and the density of the whole population. 
We look for a stable configuration, that is a Nash equilibrium: 
a configuration where, keeping into account the choices of the others, no player would spontaneously decide to change his own choice.
In an equilibrium regime, the corresponding density of the average player is stable as time goes to $+\infty$, and coincides with the  population density $m$. 

{F}rom a  PDE point of view,
this  equilibrium configuration  is characterized by a system of a fractional Hamilton-Jacobi
equation with Hamiltonian $H$ given by the Legendre transform of $L$, coupled with  a fractional stationary
Fokker-Planck equation
describing the long-time distribution of all agents, moving according to the control which minimizes the long time average cost
(see \cite{ll}, \cite{gomesbook}). 
\smallskip

We recall that MFG with jumps have been very recently considered in the literature by using a completely different approach based on 
probabilistic techniques    in  \cite{kol}, where the  theory of  non-linear Markovian propagators is used, and in \cite{ca}, where the players control the intensity of  jumps. 

\medskip 

In this paper we start with the analysis of  stationary fractional
mean field game systems, in the periodic setting,
with fractional exponent greater than~$\frac{1}{2}$.
We restrict  to this regime since  the fractional Laplacian
operator with drift  presents different properties
depending on the fact that the fractional exponent is greater
or lower than~$\frac{1}{2}$.
In the   case $s>\frac{1}{2}$, the diffusion component dominates the drift term, and so, 
the drift term can be treated as a lower-order term. Moreover  the  kernel of the  linear operator defined
 by the fractional Laplacian with drift  can be estimated in terms of the fractional heat kernel  (see \cite{bj}). 
We provide in this paper an accurate analysis of
steady state solutions to the fractional Fokker-Planck
equations in the periodic setting, with bounded drift and
fractional exponent~$s$ greater than~$\frac{1}{2}$, see
Section~\ref{sectionfp}.  
On the other hand, we discuss  in Section~\ref{submezzo}
some examples in the case of fractional Laplacian operator
with fractional exponent $s$  lower than~$\frac{1}{2}$ and bounded drift,
%showing that such operator in general does not satisfy the
%Strong Maximum Principle, 
which suggest  that the study
of fractional MFG  in the range $s<\frac{1}{2}$
presents structural differences with respect to the
range~$s>\frac{1}{2}$. 

\smallskip 

We consider the following ergodic fractional MFG on the~$N$-dimensional torus $Q:=\R^N/\Z^N$. The goal is to
find a constant~$\lambda\in\R$ for which there exists a couple $(u,m)$ solving
\begin{equation}\label{mfg2}\begin{cases}
(- \Delta)^s u+ H(\nabla u)+\lambda =f(x,m),\\
(-\Delta )^s m-\div(m \nabla H(\nabla u) )=0, \\
\int_{Q} m\, dx=1. \end{cases}
\end{equation}

Here we consider the fractional
Laplacian~$(-\Delta)^s=(-\Delta_Q)^s$ defined on the torus~$Q$
with fractional parameter~$s\in\left(\frac{1}{2},1\right)$. 
This operator can be defined directly by the multiple Fourier series
\[(-\Delta_Q)^s u(x):= \sum_{k\in \Z^N} |k|^{2s} c_k(u) e^{ik\cdot x}\] where $c_k$ are the Fourier coefficients of $u:Q\to \R$ (see \cite{rs}).
We identify functions defined on~$Q$ with their
periodic extensions to~$\R^N$,
and it is possible to show that for such functions~$u$, 
the periodic distribution~$(-\Delta_Q)^s u(x)$
coincides with the distributional fractional Laplacian
on~$\R^N$ of~$u$ (see \cite[Theorem A]{rs2}).

We shall assume that $H:\R^N\to\R$ is locally Lipschitz continuous and  strictly convex, and that
there exist some $C_H > 0$, $K>0$ and $\gamma > 1$
such that, for all $\pcors\in\R^N$,
\begin{equation}\begin{split}\label{Hass}
&C_H |\pcors|^{\gamma} - C_H^{-1} \le H(\pcors) 
\le C_H^{-1} (|\pcors|^{\gamma} + 1), \\ 
&\nabla H(\pcors)\cdot \pcors-H(\pcors)\geq C_H
|\pcors|^\gamma -K
\quad {\mbox{   and  }} \quad   |\nabla H(\pcors)|\leq C_H|\pcors|^{\gamma-1}.
\end{split}\end{equation}

As for the function~$f$, we consider both the case of local and
the case of nonlocal coupling.
We will give more precise
assumptions\footnote{ 
With a slight abuse of notation,
we write~$f[m]$ when we intend the
action of the function~$f$ to a function~$m$
and~$f(\cdot,m)$ when we intend the
map~$x\mapsto f(x,m(x))$.
The two cases are structurally different, since~$f[m]$
takes into account a ``nonlocal setting", in which,
for instance, $f[m]$
can be the convolution of~$f$ with a kernel 
(in particular, $f[m](x)$
does not depend only on~$x$ and on~$m(x)$,
but rather on~$x$ and on all the
values that~$m$ may attain).
A more precise setting
is discussed in Section~\ref{sectionreg}.} in what follows about this.  

Moreover, following \cite{trbook}, 
for~$p>1$ and~$\sigma\geq 0$,
we define the Bessel potential space $H^\sigma_p(Q)$ as
\begin{equation}\label{bessel}
H^\sigma_p(Q) := \Big\{ 
u\in L^p(Q):\ (I-\Delta)^\frac\sigma2 u\in L^p(Q)
\Big\}\qquad{\mbox{with }}\;
\|u\|_{H^\sigma_p(Q)} := \|(I-\Delta)^\frac\sigma2 u\|_{L^p(Q)}\,.
\end{equation}
In this setting, we say that a classical solution to the system \eqref{mfg2} is a
triple~$(u, \lambda, m)\in C^{2s+\theta}(Q)\times 
\R\times H^{2s-1}_p(Q)$, for all 
$\theta<2s-1$ and for all $p>1$.

\smallskip 

Our main result is the following, and it is proved in
Theorems~\ref{solmfgreg}, \ref{solmfgbdd} and~\ref{solmfg}. 

\begin{theorem} \label{TH:MAIN}
Let $s\in\left(\frac{1}{2},1\right)$. Then \eqref{mfg2} admits a classical solution in the following cases.
\begin{enumerate}
\item $\gamma>1$ and~$f:L^{1}(Q)\to \R$ maps
continuously~$C^{\alpha}(Q) $, for some~$\alpha < 2s-1$,
into a bounded subset of $W^{1,\infty}(Q)$.
\item $1<\gamma\leq 2s$ and $f:Q\times [0, +\infty)\to \R$
is continuous and bounded. 
\item $ 1< \gamma< \frac{N}{N-2s+1}$ for $N>1$,
$1<\gamma\leq 2s$ for $N=1$,  and $f:Q\times [0, +\infty)\to \R$
is locally Lipschtiz continuous and satisfies
\begin{equation}\label{gr1}
-Cm^{q-1}-K\leq  f(x, m) \leq C m^{q-1}+K,\end{equation}
for some~$C$, $K>0$ and
\begin{equation}\label{gr2}
1<q<1+\frac{(2s-1)}{N}\frac{\gamma}{\gamma-1}.\end{equation}
\end{enumerate} 
\end{theorem} 

Now, we discuss in more details the results in Theorem~\ref{TH:MAIN}. 

In the case $(1)$, that is in the case in which the coupling $f$
is a smoothing potential, we obtain existence of
solutions to the MFG system 
by taking advantadge of a classical approach given in~\cite{ll},
based on the Schauder Fixed Point Theorem.
To get the existence result in this case,
we use some estimates on the solutions to stationary
Fokker-Planck equations
obtained in Section~\ref{sectionfp}
and a-priori gradient estimates on
solutions of fractional coercive Hamilton-Jacobi equations,
inspired by the
Bernstein method in~\cite{blt}. 

As for the case of the
local coupling, we use a different approach. 
First of all, in order to get a-priori gradient estimates on
solutions of fractional coercive 
Hamilton-Jacobi equations we
cannot use anymore
the Bernstein method, since the function~$x\mapsto f(x,m(x))$
is not in general Lipschtiz continuous.
So, we use the so-called Ishii-Lions method
(see~\cite{bcci1}) to obtain gradient estimates on
solutions of fractional coercive Hamilton-Jacobi equations.
This method requires, in particular, 
that~$\gamma\leq 2s$, where~$\gamma$ is
the growth of the Hamiltonian given in~\eqref{Hass}
and~$s$ is the fractional exponent of the Laplacian.
The gradient estimates in this case 
depend only on the $L^\infty$ norm of the solutions
and of~$f$. 

In case $(2)$, this result permits to conclude the proof of Theorem~\ref{TH:MAIN},
by first regularizing  the potential and then passing to the limit. 

In case $(3)$, in which the local coupling term is unbounded, we use the variational approach, 
which goes back to the seminal work \cite{ll} (see also \cite{cgpt}, \cite{c16}): 
the MFG system is obtained (at least formally) as the optimality condition of an appropriate 
optimal control problem on the fractional Fokker-Planck equation.

First of all, the function $f(x, \cdot)$ can be unbounded
both from below and from above, so in general the
energy associated to the MFG system is not even bounded.
The condition on the growth of $f$ with respect to $m$,
given in~\eqref{gr1} and~\eqref{gr2},
is necessary to get boundedness of the energy associated to the system and then to obtain  existence of minimizers by direct methods. 

Note that our assumption allows us to treat both the case
in which the coupling is an increasing function of $m$,
that is a congestion game, in which players aim
to avoid regions where the population 
has a high density, and the opposite case in which the
coupling  is a decreasing function in $m$,
modelling a game in which every player is attracted by
regions where the density of population is high. 

Finally we point out that the condition on the growth of the
Hamiltonian in~$(3)$ of Theorem~\ref{TH:MAIN}, that is  
\[1< \gamma< \frac{N}{N-2s+1}, \]
is just a technical condition, that can be eliminated once a-priori gradient
estimates on the solutions of fractional coercive Hamilton-Jacobi equations 
depending   only  on the $L^\infty$ norm of the potential term $f$ and not on the $L^\infty$ norm of the solutions $u$ are available.  In the case of the classical Laplacian such a result has been obtained 
by an improved Bernstein method, based also on
Ishii-Lions type arguments, in \cite{clp}.
We believe that such an approach can be adapted to
the fractional case, and this will be the 
topic of future research.

\smallskip 

The paper is organized as follows. In Section~\ref{sectionfp}
we provide some results on a-priori estimates,
 existence and uniqueness of solutions to stationary fractional
Fokker-Planck equations in the periodic setting.
These results should be classical, and well known,
nevertheless due to the lack of a
precise references in the literature, we provide also a
sketch of the proofs. 
In Section~\ref{sectionHJ}, we recall the existing results
about a-priori gradient bounds for solutions to fractional
Hamilton-Jacobi equations with coercive Hamiltonians and
on the solvability of ergodic problems in this setting.
Section~\ref{sectionreg} is devoted to the analysis of MFG
systems in the case of regularizing nonlocal coupling.
Section~\ref{sectionbounded}
contains the existence result for MFG systems with local
bounded coupling. In Section~\ref{sectionunbounded},
we consider fractional MFG systems with local unbounded coupling.
Finally,
Section~\ref{sectionimprovement} contains
the improvement of regularity of solutions of the MFG system
in the case in which the coefficients are more regular, and 
the uniqueness result for increasing coupling terms. 
 
\section{Steady state solutions to fractional Fokker-Planck equations}\label{sectionfp}
We provide here some results on existence, uniqueness and
regularity of steady state solutions to fractional Fokker-Planck
equations in the periodic setting. 

First of all we recall some simple result about Bessel potential spaces. 
We recall that (see \cite{kim})
the norm~$\|\cdot\|_{H^\sigma_p(Q)}$ defined in~\eqref{bessel}
is equivalent to the norm
$$
\|u\| = \|u\|_{L^p(Q)}+\|(-\Delta)^\frac\sigma2 u\|_{L^p(Q)}\,.
$$
Observe that the space $H^\sigma_2 (Q)$ coincides with $W^{\sigma,2}(Q)$. 
Moreover we have the following embedding results. 

\begin{lemma}\label{lemmazero}
For every $\sigma\geq 0$, $p>1$
and $\eps>0$, we get 
\[ H^{\sigma+\eps}_p(Q)\subseteq W^{\sigma,p}(Q)
\subseteq H^{\sigma-\eps}_p(Q),\]
with continuous embeddings. 
Moreover 
\begin{equation}\label{eqmp}
W^{m,p}(Q)= H^m_p(Q)\qquad \textrm{if $m\in \mathbb N$}.
\end{equation}
\end{lemma}

\begin{proof} The proof of this result
is given in~\cite[Theorem~3.2]{lm}, for~$Q=\R^N$.
Then in~\cite[Section~4]{lm}, the result is extended to~$
Q=\Omega$ with~$\Omega$ bounded open set with regular
boundary, since it is proved (see~\cite[Proposition~4.1]{lm})
that~$H^\sigma_p(\Omega)$
coincides with the set of restrictions to~$\Omega$
of functions in~$H^\sigma_p(\R^N)$. 

The same argument (even simpler) permits to show
also the result for the periodic case. 
\end{proof}

\begin{lemma}\label{lemmauno}
Let $w\in H^\sigma_p(Q; \R^N)$, with $\sigma\ge 0$. 
Then there exists a unique solution 
$m\in H^{2s-1+\sigma}_p(Q)$ to the problem
\begin{equation}\label{eqm}
(-\Delta)^{s}m = \div(w),\qquad {\mbox{with }}\; \int_Q m \, dx=1.
\end{equation}
Moreover there exists $C>0$, depending on $p$,  such that 
\begin{equation}\label{esw} \|m\|_{H^{2s-1+\sigma}_p(Q)}\leq C \|w\|_{H^{\sigma}_p(Q)}.\end{equation} 
\end{lemma}

\begin{proof}
We first show that the following auxiliary problem
\begin{equation}\label{eqmu}
-\Delta u = \div(w), \qquad {\mbox{with }}\; \int_Q u\, dx=0,
\end{equation}
admits a unique solution $u\in H^{1+\sigma}_p(Q)$.

Assume first that $w$ is smooth, let $u$ be the unique smooth solution to \eqref{eqmu},
and let $v\in C^\infty(Q)$ be a test function. Multiplying \eqref{eqmu}
by $(-\Delta)^\frac{\sigma}{2}v$ and integrating by parts, we get
\begin{eqnarray*}
&& \int_Q u\, (-\Delta)^{1+\frac\sigma 2}v\,dx =
 - \int_Q (-\Delta)^\frac{\sigma}{2}w \cdot\nabla v\,dx
\\&&\qquad\quad \le \|w\|_{H^\sigma_p(Q)}\,\|\nabla v\|_{L^{p'}(Q)}\le 
C  \|w\|_{H^\sigma_p(Q)}\,\|(-\Delta)^\frac{1}{2} v\|_{L^{p'}(Q)},
\end{eqnarray*} 
where the last inequality follows from \eqref{eqmp} with $m=1$.
Here above~$p'=\frac{p}{p-1}$ is the conjugate exponent of~$p$.

As a consequence, by taking $\psi := (-\Delta)^\frac{1}{2} v$, 
which is an arbitrary test function with zero average, we get 
\[
\int_Q u\, (-\Delta)^{\frac{1+\sigma}{2}}\psi\,dx \le 
C  \|w\|_{H^\sigma_p(Q)}\,\|\psi\|_{L^{p'}(Q)}, \qquad\forall\psi\in C^\infty(Q),
\]
which implies that \begin{equation}\label{es1} \|u\|_{H^{1+\sigma}_p(Q)}\le C  \|w\|_{H^\sigma_p(Q)}.\end{equation}
The result in the general case then follows by approximating $w$ with smooth vector fields.

\smallskip

Letting now $m := 1+ (-\Delta)^{1-s}u\in H^{2s-1+\sigma}_p(Q)$, 
so that $(-\Delta)^s m = -\Delta u$,
we have that 
$$\int_Q m\, dx = 1$$ and $m$ is the (unique) solution to \eqref{eqm}. Finally, recalling the definition of $m$, we get that
\begin{equation}\begin{split}\label{jegberbger} 
&\|m\|_{H^{2s-1+\sigma}_p(Q)}=\| (I-\Delta)^{s-\frac{1}{2}+
\frac{\sigma}{2}}m\|_{L^p(Q)} \\ &\qquad \leq \| 
(I-\Delta)^{ \frac{1-\sigma}{2}}u\|_{L^p(Q)} 
+ \|(I-\Delta)^{s-\frac{1}{2}+\frac{\sigma}{2}}u\|_{L^p(Q)} 
=\|u\|_{H_p^{ \sigma+1}(Q)}
+ \|u\|_{H_p^{\sigma-1+2s}(Q)} .\end{split}\end{equation}
Notice now that~$2s-1\in (0,1)$,
therefore~$\sigma-1+2s\leq \sigma +1$,
Hence, \eqref{jegberbger}, together with~\eqref{es1}, gives~\eqref{esw},
thanks to Lemma~\ref{lemmazero}. 
\end{proof}

\begin{lemma}\label{lemmadue}
Let~$r>1$ and~$m\in L^1(Q)$ be such that $\int_Q m =1$ and
\begin{equation}\label{senzaw}
\int_Q m (-\Delta)^s \phi \le C \|\nabla\phi\|_{L^{r'}(Q)}, 
\qquad \forall \phi\in C^1_{\rm per}(\R^n),
\end{equation}
with~$r'=\frac{r}{r-1}$, for some $C>0$.
Then $(-\Delta)^{s-\frac{1}{2}}m\in L^r(Q)$ and 
\begin{equation}\label{eqstima}
\|(-\Delta)^{s-\frac{1}{2}}m\|_{L^r(Q)}\le C.
\end{equation}
\end{lemma}

\begin{proof}
By \cite{trbook} (recall also~\eqref{eqmp}), 
we know that $W^{1,r'}(Q)$ is isomorphic to $H^1_{r'}(Q)$,
so that in particular there exists a constant $C=C_{r'}>0$ such that 
\begin{equation}\label{jgjbgbg}
\|\nabla \phi\|_{L^{r'}(Q)}\leq C\|(-\Delta)^{\frac{1}{2}}\phi\|_{L^{r'}(Q)}.
\end{equation}

Let $m_\eps:=m\star\chi_\eps$, where $\chi_\eps$ is a standard  mollifier.
Then \eqref{senzaw} reads 
 \[
\int_Q m_\eps (-\Delta)^s \phi \le C \|\nabla\phi\|_{L^{r'}(Q)}, \qquad \forall \phi\in C^1_{\rm per}(\mathbb R^n).
\]
Therefore, integrating by parts and recalling~\eqref{jgjbgbg}, we obtain 
\[
\int_Q (-\Delta)^{s-\frac{1}{2}}m_\eps (-\Delta)^{\frac{1}{2}} \phi\, dx=
\int_Q m_\eps (-\Delta)^s \phi\, dx
 \le C \|\nabla\phi\|_{L^{r'}(Q)}\leq 
 C\| (-\Delta)^{\frac{1}{2}} \phi\|_{L^{r'}(Q)}, 
 \] 
 from which we obtain the desired inequality~\eqref{eqstima} for $m_\eps$
and then for $m$, letting $\eps \to 0$. 
\end{proof}

Finally we consider steady state solutions to the periodic fractional 
Fokker-Planck equation.

\begin{proposition}\label{lemmaunoemezzo}
Let $b\in L^\infty(Q;\R^N)$. Then, there exists a unique solution 
$m\in H^{2s-1}_p(Q)$, for all~$p > 1$, to the problem
\begin{equation}\label{eqkolmo}
(-\Delta)^{s}m + \div(bm) = 0,
\end{equation}
with $\int_Q m \, dx=1$, and
\[
\|m\|_{H^{2s-1}_p(Q)} \le C,
\]
where $C>0$ depends only on $N$, $p$ and $\|b\|_{L^\infty(Q;\R^N)}$.
In particular, we have that~$m\in C^\theta(Q)$,
for every~$\theta\in (0, 2s-1)$. 

Furthermore, we get that there exists a constant~$
C=C(s,N,b)>0$ such that $$0<C\leq m(x)\leq C^{-1}, \quad
{\mbox{ for any }}x\in Q.$$  
\end{proposition}

\begin{proof} 
Assume $b$ to be smooth, the general case will follow by an approximation argument.
\medskip

\noindent   {\bf Step 1: Existence and uniqueness of a solution.} 
The existence result follows by the Fredholm alternative.

More precisely, for $\Lambda$ large enough, by Lax-Milgram Theorem, the equation
\[
(-\Delta)^{s} v - b \cdot \nabla v + \Lambda v = \psi
\]
has a unique solution $u \in H^{s}_2(Q)$,
for any fixed $\psi \in L^2(Q)$.
Therefore, the mapping~$\mathcal{G}_\Lambda$,
defined by~$v = \mathcal{G}_\Lambda \psi$, 
is a compact mapping of~$L^2(Q)$ into itself.

Now, equation \eqref{eqkolmo} may be rewritten as
\begin{equation}\label{Gstar}
(I - \Lambda \mathcal{G}_\Lambda^*) m = 0.
\end{equation}
By the Fredholm alternative, the number of linearly independent solutions of \eqref{Gstar} is the same as that of the adjoint problem, that is
\[
(I - \Lambda \mathcal{G}_\Lambda) v = 0,
\]
that corresponds to
\begin{equation}\label{adjo}
(-\Delta)^{s} v - b \cdot \nabla v  = 0.
\end{equation}
Any $v\in  H^{s}_2(Q)$ solving \eqref{adjo} is in $C^{2s}(Q)$
(due to~\cite[Lemma~2.2]{pp}), and
then it must be constant by the Strong Maximum Principle (see~\cite{cio}).
We conclude that there exists~$m$
solving~\eqref{eqkolmo} (in the distributional sense),
and such~$m \in L^2(Q)$ is unique up to a multiplicative constant. 
\medskip

\noindent   {\bf Step 2: Positivity.} 
Fix a nonnegative periodic  Borel initial datum  $z_0$, and consider the following Cauchy problem, 
\begin{equation}\label{cauchypb}
\begin{cases}
\partial_t z+ (-\Delta)^{s} z - b \cdot \nabla z = 0 & \text{in $\R^N \times (0,\infty)$}, \\
z(\cdot, 0) = z_0(\cdot).
\end{cases}
\end{equation} 
We recall estimates of heat kernel of fractional Laplacian perturbed by gradient operators obtained
in~\cite[Theorem~2]{bj}: for every~$t_0>0$,
there exists a constant~$C>0$,
depending on $t_0$, $s$, $b$ and~$N$, 
such that 
\begin{equation}\label{bogdan} Cp(t,x,y)\leq p'(t,x,y)\leq
C^{-1} p(t,x,y),\quad {\mbox{for any }} x,y\in \R^N
{\mbox{ and }} t\in (0, t_0),\end{equation}
where $p(t,x,y)$ is the fractional heat kernel and~$p'(t,x,y)$
is the kernel associated to the operator~$\partial_t  + (-\Delta)^{s}
- b \cdot \nabla $.

Now we fix $x_0\in Q$ 
and we take a mollifying sequence
\begin{equation}\label{convmes} z_{0,n}\rightharpoonup \delta_{x_0}
\end{equation} in the sense of measure. 
Let $z_n$ be the solution to \eqref{cauchypb} with intial datum $z_{0,n}$. 

So, the solution $z_n$ of \eqref{cauchypb} satisfies
\begin{equation*}  z_n(x,1)\geq \tilde C  \int_Q z_{0,n}(x)dx=\tilde C,
\end{equation*} 
where $\tilde C>0$ is a constant depending on $b$, $s$ and~$N$.  
In particular, by the comparison principle,
\[z_n(x,t)\geq  \tilde C \quad  {\mbox{ for any }}  t\geq 1.\]
By \cite[Theorem 2]{bcci}, $z_n(\cdot, t) - \Lambda_n t 
- \bar{z}_n(\cdot)$ converges uniformly to zero as $t \to +\infty$, 
where the couple~$(\Lambda_n, \bar{z}_n)$
solves the stationary problem
\begin{equation}\label{ergod}
(-\Delta)^{s} \bar z_n - b \cdot \nabla \bar z_n = \Lambda_n \quad \text{in $Q$}.
\end{equation}
Note that $(\Lambda_n, \bar{z}_n)$ solving \eqref{ergod}
must satisfy $\Lambda_n = 0$, so that~$\bar{z}_n$ is
identically constant on~$Q$; hence~$z_n(\cdot, t) \to \bar{z}_n\geq
\tilde C  $ uniformly on~$Q$ as~$t \to +\infty$. 

By multiplying the equation in \eqref{cauchypb} by $m$, the equation
in~\eqref{eqkolmo} by $z_n$, and integrating by parts on~$Q$,
we obtain that, for all $t > 1$,
\[
\int_Q \partial_t z_n(x,t) m(x) dx = 0,
\]
so
\begin{equation}\label{convmes2} 
 \int_Q z_{0,n}(x) m(x) dx = \int_Q z_n(x,t) m(x) dx \to \bar{z}_n\geq \tilde C >0
\end{equation} 
as $t \to +\infty$, since~$\int_Q m \, dx=1$.
Now we send $n\to +\infty$ in~\eqref{convmes2} and we get,
recalling \eqref{convmes},
\[m(x_0)\geq \tilde C.\]
Since this is true for every $x_0\in Q$,
we get that there exists a constant~$C=C(s,N,b)>0$
such that~$0<C\leq m(x) $.  
\medskip

\noindent   {\bf Step 3: Boundedness and regularity.}
The same argument as in Step~2, using the bound from above
in~\eqref{bogdan} (instead of the bound from below), 
gives that there exists a constant~$C'=C'(s,N,b)>0$ 
such that $ m(x)\leq C'$.  

Since $b m \in L^\infty(Q)$, by Lemma \ref{lemmauno} we have
that~$m \in H^{2s-1}_p(Q)$, for all $p > 1$. 
In particular, we have that~$m\in C^\theta (Q)$, for every~$\theta\in (0,2s-1)$.  
\medskip
\end{proof}

\subsection{The case $s<\frac{1}{2}$} \label{submezzo}
Note that in the case $s<\frac{1}{2}$  the solutions
to 
$$(-\Delta)^s m=\div w,$$
for $w\in H^{\sigma}_p$, have to be intended in some
weak sense. In particular, if~$\sigma<1-2s$,
the solution $m$ is a distribution. 

Moreover the associated kernel of the operator $(-\Delta)^s +b(x)\cdot \nabla$   is not bounded from below 
by the  fractional heat kernel, and  it does not produce strictly 
positive solutions. These phenomena will be discussed in details in the following remarks.  
This suggests that the study of fractional Mean Field Games in the range $s<\frac{1}{2}$ 
presents structural differences with respect to the range $s>\frac{1}{2}$. 
\begin{remark}\label{RK1}
Concerning the optimality of
the regularity results in Proposition~\ref{lemmaunoemezzo},
we point out that
the solution~$m$ may vanish at a point
and is not better than~$C^{2s}$ for~$s\in(0,\,1/2)$.
To see a one-dimensional example in~$\R$,
we take~$n=1$, $s\in(0,\,1/2)$ and
$$ b(x):=-\frac{1}{m(x)} \int_0^x (-\Delta)^s m(y)\,dy,$$
with~$m(x):=|x|^\theta$, with~$\theta\in(2s,1)$.

Using the substitution~$z=y/x$, we see that
\begin{equation}\label{SCAL} \int_{\R} \frac{|x+y|^\theta+|x-y|^\theta-2|x|^\theta}{|y|^{1+2s}}\,dy=
|x|^{\theta-2s}
\int_{\R} \frac{|1+z|^\theta+|1-z|^\theta-2|z|^\theta}{|z|^{1+2s}}\,dz\end{equation}
and so~$(-\Delta)^s m(x)=-c|x|^{\theta-2s}$, for some~$c>0$.

This setting gives that, for small~$|x|$,
$$ |b(x)|\le \frac{C}{|x|^\theta} \int_0^{|x|} |y|^{\theta-2s}\,dy
\le C |x|^{1-2s},$$
up to renaming~$C>0$.
This gives that~$b$ is locally bounded (and H\"older continuous
with exponent~$1-2s$).
Moreover, we have that
$$ {\rm div}(bm)= (bm)' =-\left( 
\int_0^x (-\Delta)^s m(y)\,dy\right)'=
-(-\Delta)^s m,$$
hence the equation is satisfied.
\end{remark}

\begin{remark}
The example of Remark~\ref{RK1} can also be used to
show that the positivity results of~\cite{bj} do not hold in general for~$s\in(0,\,1/2)$.
For instance,
we take~$n=1$, $s\in(0,\,1/2)$,
$v(x):=|x|^\theta$, with~$\theta\in(2s,1)$.

{F}rom~\eqref{SCAL}, we know that~$(-\Delta)^s v(x)=-c|x|^{\theta-2s}$, for some~$c>0$.
So we define~$ b(x):= -\frac{c}\theta\,|x|^{-2s}x$
and we notice that~$b$ is locally bounded (and H\"older continuous
with exponent~$1-2s$) and
$$ (-\Delta)^{s} v - b \cdot \nabla v  = 
-c|x|^{\theta-2s} +\left(\frac{c}\theta\,|x|^{-2s}x\right)\cdot(\theta |x|^{\theta-2}x)
=
0,$$
hence~\eqref{adjo} is satisfied.

Since~$v\ge0$ but~$v(0)=0$, this example shows that the strong maximum principle is
violated in this case. Note that~$v$ solves a.e. the
equation, but it is not a viscosity (sub)solution
of the equation at~$x=0$. Indeed, 
by the strong maximum principle proved
in Lemma~4.4 in~\cite{blt}, the unique viscosity solutions
to~$(-\Delta)^{s} v - b \cdot \nabla v$ are constants. 

Moreover, $v$ is also a (stationary) solution of the heat flow
associated to~\eqref{adjo}, corresponding to an initial datum which is nonnegative
and that does not become strictly positive as time flows (this lack of positivity gain in time
can be seen as a counterpart when~$s\in(0,\,1/2)$ to the positivity of the heat kernel
established in~\cite{bj}).
\end{remark}

\section{Fractional Hamilton-Jacobi equations with coercive Hamiltonian}\label{sectionHJ} 

We collect some results on the
Lipschitz continuity of viscosity solutions to Hamilton-Jacobi equations and on the solution to the ergodic problem. 
There are different kinds of results, depending on the fact that
the Hamiltonian term is dominant or not with respect to the fractional Laplacian term. 
Actually, the growth~$\gamma\leq 2s$ allows the use of the so called  Ishii-Lions method, in particular getting a-priori  estimates on the gradient of the solution
which depend only on the~$L^\infty$ norm of the 
potential term, whereas in the case~$\gamma>2s$
the Bernstein method is used, obtaining a-priori
estimates on the gradient of the solution
which depend only on the Lipschitz norm of the 
potential term.

We consider the following Hamilton-Jacobi equation 
\begin{equation}\label{eqHJ} 
(-\Delta)^s u+ H(\nabla u)+\lambda= f(x), \qquad x\in \R^N.
\end{equation} 

We assume that $f\in C(\R^N)$,  and that~$f$ is $\Z^N$-periodic.  

\begin{theorem}\label{ergodic} 
Let $s>\frac{1}{2}$ and $\gamma\leq 2s$. 

Then the following statements hold. 
\begin{enumerate} 
\item If $u$ is a continuous periodic solution to \eqref{eqHJ}, then
there exists a constant $K>0$, depending on~$\|u\|_{L^\infty(Q)}$,  $\|f\|_{L^\infty(Q)}$ 
and~$|\lambda|$, such that
\[\|\nabla u\|_{L^\infty(Q)}\leq K.\]
Moreover, there exists a constant $C>0$,
depending only on the period of $u$ and 
$\|f\|_{L^\infty(Q)}$, such that $\|u\|_{L^\infty(Q)}\leq C$.

\item 
There exists a unique constant $\lambda\in\R$
such that~\eqref{eqHJ} has a periodic solution $u\in W^{1, \infty}(\R^N)$ and 
\begin{equation}\label{defbwb}
\lambda= \sup \{c\in\R\; {\mbox{ s.t. }}\; \exists u\in W^{1, \infty}(\R^N) 
{\mbox{ s.t. }} (-\Delta^s) u+ H(\nabla u)+ c\leq f(x)\}.\end{equation} 
Moreover, $u$ is the unique Lipschitz viscosity solution to \eqref{eqHJ}
up to addition of constants. 
\end{enumerate} 

Finally, if $f\in C^{\theta}(\R^N)$, for some $\theta\in (0,1]$, 
then~$u\in C^{2s+\alpha}(Q)\cap H^{2s}_p(Q)$,
for every $\alpha<\theta$ and every~$p>1$. 
\end{theorem} 

\begin{proof} 
The a-priori estimates on the gradient is proved in~\cite[Theorem 2]{bcci1}.
The a-priori estimate on the $L^\infty$ norm of $u$ can be obtained as  in \cite[Proposition1]{bcci1} or \cite[Lemma 4.2]{blt}.

The existence of $\lambda$ and of a 
unique (up to constants) viscosity solution to~\eqref{eqHJ}
is given in~\cite[Theorem~1]{bcci}.
Formula~\eqref{defbwb} can be proved by a standard argument,
using the Strong Maximum Principle, which holds for operators
as~$(-\Delta)^s +b\cdot \nabla$, with~$s>\frac{1}{2}$
and~$b$ continuous (see~\cite{cio}).

Finally, if $f$ is H\"older continuous, applying~\cite{cs},
we get that $u\in C^{1+\alpha}(Q)$ for any $\alpha<2s-1$, and finally,
by the bootstrap argument in~\cite[Theorem~6]{bfv}, 
we obtain the desired regularity. 

Also, since $(-\Delta)^s u\in L^\infty(Q)$, 
then by~\cite[Theorem 2.1]{kim} it follows
that~$u\in H^{2s}_p(Q)$ for every $p>1$.  
\end{proof} 

\begin{theorem}\label{ergodic2} 
Let  $\gamma>1$ and assume that $f\in W^{1, \infty}(\R^N) $. 

Then the following statements hold. 
\begin{enumerate} 
\item If $u$ is a continuous solution to \eqref{eqHJ},
then there exists a constant $K>0$, depending 
on~$\|f\|_{L^\infty(Q)}$, $\|\nabla f\|_{L^\infty(Q;\R^N)}$
and~$|\lambda |$, such that
\[\|\nabla u\|_{L^\infty(Q)}\leq K.\]
\item 
There exists a unique constant $\lambda\in\R$
such that~\eqref{eqHJ} has a periodic solution $u\in W^{1, \infty}(\R^N)$ and 
\[\lambda= \sup \{c\in\R\; {\mbox{ s.t. }}\; \exists u\in W^{1, \infty}(\R^N) 
{\mbox{ s.t. }} (-\Delta^s) u+ H(\nabla u)+ c\leq f(x)\}.\]
Moreover $u$ is the unique Lipschitz viscosity solution to \eqref{eqHJ}
up to addition of constants.
\end{enumerate} 
Finally, $u\in C^{2s+\alpha}(Q)\cap H^{2s}_p(Q)$, for every~$\alpha<1$
and every~$p>1$. 
\end{theorem} 

\begin{proof} The a-priori estimates on the gradient is proved in~\cite[Theorem 3.1]{blt}. 
The existence of $\lambda$ and of a 
unique (up to constants) viscosity solution to~\eqref{eqHJ}
is given in~\cite[Proposition 4.1]{blt}.  

As for the rest we proceed  analogously as in Theorem \ref{ergodic}. 
\end{proof} 

\begin{remark}[Case $s\leq \frac{1}{2}$] \upshape  
We note that Theorem~\ref{ergodic2} holds for
every~$s\in (0,1)$. 
As for Theorem~\ref{ergodic},
it can be proved that,
if $\gamma<2s<1$, solutions to~\eqref{eqHJ}
are actually H\"older continuous, with
H\"older exponent striclty less
than~$\frac{2s-\gamma}{1-\gamma}$
(see Remark~1 in~\cite{bcci}).
In this case, the uniqueness of the ergodic
constant~$\lambda$ remains true, 
but it is not clear anymore that the solutions of the ergodic problem are unique up to an additive constant. 
 \end{remark} 
\section{Regularizing coupling}\label{sectionreg}

The aim of this section is to prove existence of solutions for~\eqref{mfg2}
in the case~$(1)$ of Theorem~\ref{TH:MAIN}. 
For this, we consider the system
\begin{equation}\label{mfg_reg}\begin{cases}
(- \Delta)^s u+ H(\nabla u)+\lambda = f[m](x),\\ 
(-\Delta )^s m-\div(m \nabla H(\nabla u) )=0, \\   \int_{Q} m\, dx=1, \end{cases}
\end{equation}
where $f : C^{\alpha}(Q) \to W^{1,\infty}(Q)$, with~$\alpha < 2s-1$, 
is a regularizing functional.   Let
\begin{equation}\label{Xdef}
X := \left\{ m \in C^{\alpha}(Q) : m \ge 0, \,  \int_{Q} m\, dx=1\right\}.
\end{equation}
We suppose that
\begin{equation}\label{assFnonlocal}
 \text{$f$ maps continuously $X$ into a bounded set of $W^{1,\infty}(Q)$}.
\end{equation}

A typical example of $f$ satisfying \eqref{assFnonlocal} is~$f[m](x) := g(x,K \star m (x))$, where $K : \R^N \to \R$ is a Lipschitz kernel and $g:\R^N\times \R\to\R$ is a Lipschitz function,   which is $\Z^N$-periodic in $x$. 

\begin{theorem}\label{solmfgreg}
Assume that \eqref{assFnonlocal} holds. Then, there exists a classical solution $(u, \lambda, m)$  
 to the mean field game system \eqref{mfg_reg}.
\end{theorem}

\begin{proof} The statement follows by
the Schauder Fixed Point Theorem in~$X$
(we will follow the lines of~\cite[Section~3]{c14}).
More precisely, we construct a compact map~$T: X\to X$,
with~$T(m)=\mu$, as follows.  

For any $m \in X$, we consider the problem 
\begin{equation}\label{Fdef}(- \Delta)^s u+ H(\nabla u)+\lambda =f[m](x).\end{equation}
By Theorem~\ref{ergodic2}, since $f[m]$ is a Lipschitz function,
we get that there exists a unique solution~$(u, \lambda)\in 
W^{1,\infty}(Q)\times \R$. 
This implies, in particular, that $(-\Delta^s)u\in L^\infty(Q)$,
so $u\in H^{2s}_p(Q)$ for all $p>1$, thanks to~\cite{kim}.
Hence, by a bootstrap argument we get that $u\in C^{2s+\theta}(Q)$,
for all $\theta<1$. 

Now, we observe that $\|\nabla H(\nabla u)\|_{L^\infty(Q)}\leq C$, 
with a constant~$C>0$ independent of $m$, in virtue of Theorem~\ref{ergodic2}.

Let $\mu $ be the solution to 
\begin{equation}\label{eggregre}
\begin{cases} (-\Delta )^s \mu-\div(\mu \nabla H(\nabla u) )=0, \\   \int_{Q} \mu\, 
dx=1.\end{cases}
\end{equation}
By Lemma \ref{lemmaunoemezzo}, there exists a unique $\mu$ 
solution to~\eqref{eggregre}, and $\mu\in H^{2s-1}_p(Q)$ for all $p>1$,
with  $$\|\mu\|_{H^{2s-1}_p(Q)}\leq C\|\nabla H(\nabla u)\|_{L^\infty(Q)}.$$
Now, by Sobolev embedding, $\|\mu\|_{C^{\beta}(Q)}$ is bounded, for some $\alpha < \beta < 2s-1$.
So, $T: m \mapsto \mu$ is a compact mapping of $X$ into itself. 

Therefore, we only need to show that~$T$ is also continuous 
to conclude the existence of a fixed point, 
that in turn provides the existence of a solution to~\eqref{mfg_reg}. 
This follows by stability of the equation in~\eqref{Fdef}.
Indeed, for a given $f[m]$, the couple solving the first equation in~\eqref{Fdef} is unique, if we impose for example $u(0) = 0$
(see Theorem~\ref{ergodic2}). 
\end{proof}

\section{Local bounded coupling}
\label{sectionbounded}

Here we prove Theorem~\ref{TH:MAIN} under the assumptions
of case~$(2)$. For this, we now specify the setting in which we work. 

We assume that $f:\R^N\times [0, +\infty)\to\R $ 
is  a continuous function, $\Z^N$-periodic 
in $x$, that is $f(x+z, m)= f(x,m)$ for all $z\in \Z^N$, all~$x\in\R^N$ 
and all~$m\in [0, +\infty) $.
Moreover, we assume that  there exists  $K>0$ such that
\begin{equation}\label{assFlocal1}   |f(x,m)|\leq K \quad
\forall m\geq 0.   
\end{equation}
We also suppose that
\begin{equation}\label{assFlocal1BIS} 
1<\gamma\leq 2s.\end{equation}

In this framework, 
we get the following existence result, based on a
regularization argument and on the existence result given
in Theorem~\ref{solmfgreg}.
 
\begin{theorem}\label{solmfgbdd}
Under assumptions~\eqref{assFlocal1} and~\eqref{assFlocal1BIS},
there exists a classical solution $(u,\lambda, m)$
to the mean field game system~\eqref{mfg2}. 
\end{theorem} 

\begin{proof} 
We consider the following  regularization of the system \eqref{mfg2}: 
\begin{equation}\label{mfgeps}\begin{cases}
(- \Delta)^s u + H(\nabla u )+\lambda  =f_\eps[m](x),\\
(-\Delta )^s m -\div(m  \nabla H(\nabla u ) )=0, \\   \int_{Q} m \, dx=1, \end{cases}
\end{equation}
where
\[f_\eps[m](x)= f(x, m \star\chi_\eps)\star \chi_\eps (x)
=\int_Q \chi_\eps(x-y) f\left (y, \int_Q m (z)
\chi_\eps(y-z)dz\right)dy\]
and~$\chi_\eps$, for~$\eps>0$,
is a sequence of standard mollifiers.

Note that $f_\eps$ satisfies assumption~\eqref{assFnonlocal}
(see e.g.~\cite[Example 5]{c14}), and
therefore, for every $\eps>0$, 
there exists a classical solution~$(u_\eps, \lambda_\eps, m_\eps)$ 
to~\eqref{mfgeps}, thanks to Theorem~\ref{solmfgreg}.

Now, let $\overline x_\eps$ and $\underline x_\eps$ be
such that $u_\eps(\overline x_\eps)=\max u_\eps$
and~$u_\eps(\underline  x_\eps)=\min u_\eps$.
Evaluating the Hamilton-Jacobi equations at these points,
we get 
$$f_\eps[m_\eps](\underline x_\eps)-C_H^{-1}\leq \lambda_\eps
\leq f_\eps[m_\eps](\overline x_\eps)+C_H^{-1},$$
where~$C_H$ is the constant given in~\eqref{Hass}.
This and~\eqref{assFlocal1} imply that~$|\lambda_\eps|\leq \tilde C$,
for some~$\tilde C>0$.
So, up to passing to a subsequence,
we can assume that~$\lambda_\eps\to \lambda$, as~$\eps\to 0$. 

Again by assumption \eqref{assFlocal1},
using Theorem \ref{ergodic}, we get that there
exists~$C>0$, independent of $\eps$, such that  
\begin{equation}\label{ehwfhfv}
\|\nabla u_\eps\|_{L^\infty(Q)}\leq C.\end{equation} 
Hence, since $u_\eps$ solves~\eqref{mfgeps},
we get that~$(-\Delta)^s u_\eps$ is uniformly bounded
in~$L^\infty(Q)$, and so $u_\eps\in H^{2s}_p(Q)$,
for every~$p>1$, and~$\|u_\eps\|_{H^{2s}_p(Q)}\leq C$,
with $C>0$ independent of $\eps$. 

Therefore, by Sobolev embedding, we obtain
also that the sequence~$u_\eps$ is equibounded
in~$C^{1+\alpha}(Q)$, for some~$\alpha\in (0,1)$. 

Moreover, the estimate in~\eqref{ehwfhfv} and~\eqref{Hass}
imply that 
$$\|\nabla H(\nabla u_\eps)\|_{L^\infty(Q)}\leq C,$$
for some costant $C>0$. Then, we are in the position to apply
Proposition~\ref{lemmaunoemezzo}, and conclude that, 
for all~$p>1$ and~$\alpha\in (0, 2s-1)$,
there exist constant $C_1$, $C_2>0$, depending on~$K$
and on~$p$ and~$\alpha$, respectively, such that 
\begin{equation}\label{stimeun2}
\|m_\eps\|_{H^{2s-1}_p(Q)}\leq C\qquad {\mbox{ and }}\qquad
\|m_\eps\|_{C^{\alpha}(Q)}\leq C'. 
\end{equation}
This implies that, up to subsequences, 
$m_\eps\to m$ in $H^{2s-1}_p(Q)$, as~$\eps\to 0$,
for all~$p>1$ (and also uniformly in $C^{\alpha}(Q)$ for
every~$\alpha<2s-1$, thanks to Sobolev embeddings). 
Therefore~$f_\eps[m_\eps](x)$ is equibounded
in~$C^{\alpha}(Q)$, for every~$\alpha<2s-1$. 
Since $u_\eps$ solves~\eqref{mfgeps},
this, in turn, gives that~$u_\eps$ are equibounded
in~$C^{2s+\alpha}(Q)$, for some~$\alpha\in(0,1)$. 
Therefore, by Ascoli-Arzel\`a Theorem,
we can extract a converging subsequence~$u_\eps\to u$
in~$C^{2s}(Q)$, as~$\eps\to0$. 

Note that the convergences obtained are sufficiently strong to
pass to the limit in the equations, and so we conclude
that~$(u, \lambda, m)$ is
a classical solution to \eqref{mfg2}. 
This completes the proof of Theorem~\ref{solmfgbdd}.
\end{proof} 

\section{Local unbounded coupling} \label{sectionunbounded} 

As in Section~\ref{sectionbounded},
we consider the case in which the coupling~$f$ is local,
that is, we suppose that~$f:\R^N\times [0, +\infty)\to\R $ 
is locally Lipschitz continuous in both variables,
and~$\Z^N$-periodic 
in $x$, namely $f(x+z, m)= f(x,m)$ for all $z\in \Z^N$,
all~$x\in\R^N$ 
and all~$m\in [0, +\infty) $.

Differently from Section~\ref{sectionbounded}, here~$f$
can be unbounded, both from above and from below.
Nevertheless,  we have to restrict the condition on
the growth~$\gamma$ of the Hamiltonian~$H$.
In particular, we assume 
that  there exist $C>0$ and $K>0$ such that
\begin{equation}\begin{split} \label{assFlocal}  
&-Cm^{q-1}-K\leq  f(x, m) \leq C m^{q-1}+K,  \quad 
\text{ with } 1<q<1+\frac{(2s-1)}{N}\frac{\gamma}{\gamma-1},  \\ 
{\mbox{and }}& \left\{\begin{matrix}
1<\gamma< \frac{N}{N-2s+1} &\text{ if $N>1$},\\ 
1<\gamma\leq 2s   &\text{ if $N=1$}. \end{matrix}\right.
\end{split}\end{equation} 
Note that if $N>1$ then~$ \frac{N}{N-2s+1}<2s$,
in virtue of~\eqref{assFlocal}. 

We also remark that the bound on~$\gamma$ in the case~$N>1$
is just a technical assumption, that could be removed if the a  priori bounds on the gradients
of solutions to fractional Hamilton-Jacobi equations with coercive Hamiltonian stated in Theorem \ref{ergodic} could be shown to be independent on  
the $L^\infty$ norm of the solutions. 

To provide existence of a solution to~\eqref{mfg2} in this setting,
we follow the variational approach, see~\cites{ll, cgpt, c16, bc}. 
More precisely, we associate to the mean field game system 
an energy whose minimizers will be used to construct solutions to~\eqref{mfg2}. 

\subsection{The energy associated to the system}
We denote by~$\tilde L$ the Legendre transform of $H$, i.e.
\[
\tilde L(\qcors) :=  \sup_{\pcors \in \R^N} [\pcors \cdot \qcors - H(\pcors)], \qquad{\mbox{ for any }} \qcors
\in \R^N.
\]
Observe that, by~\eqref{Hass}, there exists $C_L>0$ such that
\begin{equation}\label{Lass}
C_L |\qcors|^{\gamma'} - C_L^{-1} \le L(\qcors) \le C_L^{-1} (|\qcors|^{\gamma'} + 1)\,, 
\quad {\mbox{ for any }} v \in \R^N,
\end{equation}
where $\gamma'=\frac{\gamma}{\gamma-1}$ is the conjugate exponent of $\gamma$. 
Note that, by assumption \eqref{assFlocal},
\begin{equation}\label{assgamma} \gamma'>\frac{N}{2s-1}.\end{equation} 
We let \begin{equation}\begin{split}
\label{kcalconstraint}
\mathcal{K}: =\,&\left\{ (m,w) \in L^1(Q)\cap W^{1,\gamma'} (Q)
\times L^{\gamma'}(Q)\; {\mbox{ s.t.}} \right. \\ &\quad
\int_{Q}  m  (-\Delta)^s \varphi \, dx = \int_{Q} w \cdot \nabla \varphi \, dx 
\quad \forall \varphi \in C_0^\infty(Q), \\ &\quad \left.
\int_{Q} m \, dx = 1, \quad \text{$m \ge 0$ a.e.}  \right\}.  
\end{split}\end{equation}
We associate to the mean field game \eqref{mfg2}  the following energy \begin{equation}\label{energia}
\mathcal{E}(m, w) := \begin{cases}
\displaystyle \int_{Q} m L\left(-\frac{w}{m}\right)+F(x,m) \, dx & \text{ if $(m,w) \in \mathcal{K}$}, \\
+\infty & \text{otherwise},
\end{cases}
\end{equation}
where 
\begin{equation}\label{dati}\begin{split}&
L\left(-\frac{w}{m}\right):=\begin{cases}  \tilde L\left(-\frac{w}{m}\right) & {\mbox{ if }}m>0,\\ 
0 & {\mbox{ if }}m=0, w=0,\\ +\infty & \text{otherwise} \end{cases} \\ {\mbox{and}}
\qquad&
F(x,m):=\begin{cases} \displaystyle
\int_0^m f(x,n)\,dn& {\mbox{ if }}m\geq 0,\\ +\infty& {\mbox{ if }}
m<0.\end{cases}\end{split}\end{equation} 
Note that, since $\tilde L$ is the Legendre transform of $H$, we have that, for all $m\geq 0$,
\begin{equation}\label{leg}
mH(p)=\sup_{w} \left[-p\cdot w-m L\left(-\frac{w}{m}\right)\right].
\end{equation} 
Moreover, recalling \eqref{Lass}, we get that \begin{equation}\label{ell} C_L \frac{|w|^{\gamma'}}{m^{\gamma'-1}} -C_L^{-1} m 
\leq L\left(-\frac{w}{m}\right)\leq C_L^{-1} \frac{|w|^{\gamma'}}{m^{\gamma'-1}} +C_L^{-1} m. \end{equation}

Now, we provide a-priori estimates for couples $(m,w)\in\mathcal{K}$ with finite energy.

\begin{proposition} \label{pstime}  
Assume that $(m,w)\in \mathcal{K}$ such that 
there exists $K>0$ with 
\[E:=\int_{Q} \frac{|w|^{\gamma'}}{m^{\gamma'-1}} dx\leq K.\] 

Then, there exist $\delta>0$ and $C>0$  such that 
\begin{equation}\label{rstima} 
\|m\|_{L^q(Q)}^{q(1+\delta)} \leq C\int_{Q} \frac{|w|^{\gamma'}}{m^{\gamma'-1}} dx\leq CK
\end{equation} where $q$ is as in \eqref{assFlocal}.

Moreover, for every~$\alpha\in \left(0, 2s-1-
\frac{N}{\gamma'}\right)$, there exists a constant $C>0$,
depending on $\alpha$, such that 
\begin{equation}\label{mstima} 
\|m\|_{C^{\alpha}(Q)}  \leq C\int_{Q} \frac{|w|^{\gamma'}}{m^{\gamma'-1}} dx\leq CK.
\end{equation} 
\end{proposition} 

\begin{proof}
Note that since $m\in W^{1,\gamma'}(Q)$ and $\gamma'$
satisfies~\eqref{assgamma}, then $m\in L^p(Q)$,
for every~$p\in (1, +\infty]$. Moreover, by Sobolev embeddings,
we have that~$m\in C^{\alpha}(Q)$, for 
every~$\alpha\in \left(0, 2s-1-\frac{N}{\gamma'}\right)$.

Now, let $p>1$ and define $r_p$ as follows
\begin{equation}\label{r} \frac{1}{r_p}=\frac{1}{\gamma'}+\left(1-\frac{1}{\gamma'}\right)\frac{1}{p}.\end{equation} 
In this way, we see that~$r_p< \min\{p, \gamma'\}$. 

By~\eqref{kcalconstraint}, we have that
\begin{eqnarray*}
&&\int_Q m(-\Delta)^s \phi\, dx= \int_Q w\cdot \nabla \phi\, dx \leq 
\int_Q  \left(\frac{|w|^{\gamma'}}{m^{\gamma'-1}}\right)^{\frac{1}{\gamma'}}
m^{\frac{1}{\gamma}}|\nabla \phi|\,dx \\
&&\qquad \leq
\left(\int_Q  \frac{|w|^{\gamma'}}{m^{\gamma'-1}} \,dx \right)^{\frac{1}{\gamma}}
\|m\|_{L^p(Q)}^{\frac{1}{\gamma}}\|\nabla \phi\|_{L^{r'_p}(Q)} \leq E^{\frac{1}{\gamma'}} \|m\|_{L^p(Q)}^{\frac{1}{\gamma}}\|\nabla \phi\|_{L^{r'_p}(Q)},
\end{eqnarray*} 
for any~$\phi\in C^\infty_0(Q)$.
Here above we used the notation~$r_p'=\frac{r_p}{r_p-1}$.

Therefore, by Lemma \ref{lemmadue} we get that 
\begin{equation}\label{fbbbrb}
\|(-\Delta)^{s-\frac{1}{2}}m\|_{L^{r_p}(Q)}\le C
E^{\frac{1}{\gamma'}} \|m\|_{L^p(Q)}^{\frac{1}{\gamma}},
\end{equation}
for some~$C>0$.
Moreover, by interpolation, we get that 
\begin{equation}\label{fbbbrb2}
\|m\|_{L^{r_p}(Q)} \leq \|m\|_{L^p(Q)}^{\frac{1}{\gamma}} 
\|m\|_{L^1(Q)}^{\frac{1}{\gamma'}}=
\|m\|_{L^p(Q)}^{\frac{1}{\gamma}}.\end{equation} 
{F}rom~\eqref{fbbbrb} and~\eqref{fbbbrb2}, we conclude that
\begin{equation}\label{starstar}
\|m\|_{H^{2s-1}_{r_p}(Q)}\le C (E^{\frac{1}{\gamma'}} +1)\|m\|_{L^p(Q)}^{\frac{1}{\gamma}}.
\end{equation}

Now, we prove~\eqref{rstima}. For this, let $r=r_q$,
that is, in \eqref{r} we choose $p=q$, where $q$ is as in~\eqref{assFlocal}. 
Let $r^\star$ be such that 
\begin{eqnarray*}
&& \frac{1}{r^\star}=\frac{1}{r}-\frac{2s-1}{N} \qquad {\mbox{if }}
r< \frac{N}{2s-1},\\
{\mbox{and }} && r^\star=+\infty \qquad  {\mbox{if }}
r\geq \frac{N}{2s-1}.\end{eqnarray*} 
Notice that by \eqref{assFlocal}, it is easy to see that
\begin{equation}\label{qrstar}
q< r^\star.\end{equation} 
Therefore, by Sobolev embedding, there exists~$C>0$
such that
$$\|m\|_{H^{2s-1}_r(Q)}\geq C\|m\|_{L^q(Q)},$$ 
and so, substituting in \eqref{starstar} we get that
\[\|m\|_{L^q(Q)} \le C (E +1)\qquad\text{ and }\qquad
\|m\|_{H^{2s-1}_{r}(Q)}\leq C(E+1).\] 
Note that, in virtue of~\eqref{qrstar},
by interpolation and using~\eqref{starstar}, we get
\[\|m\|_{L^q(Q)}\leq\|m\|_{L^1(Q)}^{1-\theta}\|m\|_{L^{r^\star}(Q)}^\theta\leq\|m\|_{L^1(Q)}^{1-\theta}\|m\|_{H^{2s-1}_r(Q)}^\theta\leq C(1+E^{\frac{\theta}{\gamma'}})\|m\|_{L^q(Q)}^{\frac{\theta}{\gamma}}, \]
where $\theta$ is such that 
\[\frac{1}{q}=1-\theta+\frac{\theta}{r^\star}.\]
It is easy to check that 
\begin{equation}\label{teta} \frac{1}{\theta} = 1-\frac{1}{\gamma'} +\frac{2s-1}{N} \frac{q}{q-1}.\end{equation}
We then obtain that 
\begin{equation}\label{jgerjgerbj}
\|m\|_{L^q(Q)}^{\left(1-\frac{\theta}{\gamma}\right)
\frac{\gamma'}{\theta}} \leq  C(1+E). \end{equation}
Using \eqref{teta} we check that 
\[ \left(1-\frac{\theta}{\gamma}\right)\frac{\gamma'}{\theta}=\gamma' \frac{2s-1}{N}\frac{q}{q-1}= (1+\delta)q,\]
where 
\[\delta=\frac{1}{q-1}\left( \frac{\gamma' 
(2s-1)+N}{N}-q\right)>0\] 
thanks to~\eqref{assFlocal}. This and~\eqref{jgerjgerbj}
imply~\eqref{rstima}, as desired.

\smallskip 

We prove now \eqref{mstima}. 
In virtue of~\eqref{assgamma},
we can choose~$p$ in~\eqref{r} sufficiently large such
that 
$$\frac{N}{2s-1}<r_p<\gamma'.$$
So, from~\eqref{starstar}, using Sobolev embeddings and
reasoning as above, we obtain that 
\begin{equation}\label{okpuygm}
\|m\|_{L^p(Q)} \le C (E +1)\qquad\text{ and }\qquad  
\|m\|_{H^{2s-1}_{r_p}(Q)}\leq C(E+1).\end{equation}
{F}rom the second estimate in~\eqref{okpuygm} and
Sobolev embeddings, we obtain~\eqref{mstima}. 
This completes the proof of Proposition~\ref{pstime}.
\end{proof} 

Using the previous estimates in Proposition~\ref{pstime},
we deduce the existence of a minimizer of the energy in
the class~$\mathcal{K}$ introduced in~\eqref{kcalconstraint}. 

\begin{theorem} \label{exthm} 
There exists $(m,w)\in \mathcal{K}$ such that 
$$ \mathcal{E}(m,w)=\min_{(m,w)\in \mathcal{K}}\mathcal{E}\,.$$
\end{theorem}

\begin{proof}
First of all observe that, by Proposition \ref{pstime}
and~\eqref{ell}, there exists $C>0$ such that, for every $(m,w)\in \mathcal{K}$,
\[\mathcal{E}(m,w)\geq C\|m\|_{L^q(Q)}^{(1+\delta)q} -C+\int_Q F(x,m)dx\,.\]
{F}rom this, recalling assumption \eqref{assFlocal} and the definition of $F$ in \eqref{dati},
we conclude that there exists a constant $K$, depending on $q$, such that 
\[\mathcal{E}(m,w)\geq C\|m\|_{L^q(Q)}^{(1+\delta)q} -C' \|m\|_{L^q(Q)}^{q}-C'\geq K.\]

Let $e:=\inf_{(m,w)\in\mathcal{K}} \mathcal{E}(m,w)$. 
We fix a minimizing sequence $(m_n,w_n)$. 
Therefore $\mathcal{E}(m_n,w_n)\leq e+1$, for every $n$ sufficiently large. 
Note that by our definition of $L$ this implies that~$w_n=0$ 
where~$m_n=0$. 

Therefore, again by assumption \eqref{assFlocal}, \eqref{ell} and Proposition \ref{pstime}, we get
\begin{eqnarray*} 
\int_Q  \frac{|w_n|^{\gamma'}}{m_n^{\gamma'-1}} 
\,dx&\leq & C_L^{-1}(e+1-\int F(x,m_n)dx))\\&\leq & 
C_L^{-1}(e+1-C+C\|m_n\|^q_{L^q(Q)})
\\ &\leq &C_L^{-1}\left(e+1+C'+ K  \left(\int_Q  
\frac{|w_n|^{\gamma'}}{m_n^{\gamma'-1}} \,dx+1
\right)^{\frac{1}{1+\delta}}\right).\end{eqnarray*} 
This implies in particular that~$
\left(\int_Q  \frac{|w_n|^{\gamma'}}{m_n^{\gamma'-1}} 
\,dx\right)$ is equibounded in $n$.  

By \eqref{mstima}, this implies that~$\|m_n\|_{C^{\alpha}(Q)}
\leq C$, for some $\alpha\in(0,1)$.
Therefore, up to a subsequence,
\[m_n\to m\quad \text{ uniformly in $Q$, as $n\to+\infty$.}
%\qquad m_n\rightharpoonup m \quad\text{weakly in $H^{2s-1}_r(Q)$}.
\] 
Therefore, we have that~$m_n\to m$ in $L^1(Q)$
and $\int_Q mdx=1$.

In particular, we see that~$0\leq m_n\leq C$, for every $n$,
and then 
\[\int_Q |w_n|^{\gamma'}dx\leq  C^{\gamma'-1} \int_Q  \frac{|w_n|^{\gamma'}}{m_n^{\gamma'-1}} \,dx. \] 
This implies that $w_n$ is equibounded in $L^{\gamma'}(Q)$
and so, up to a subsequence,
\[w_n\to w\quad \text{weakly in $L^{\gamma'}(Q)$, as $n\to+\infty$}.\] 

Note that the convergences are strong enough to assure
that~$( m, w)\in\mathcal{K}$. 
Then, the desired result follows from
the lower semicontinuity of the kinetic part of the functional
and by the strong convergence in $L^q(Q)$ of $m_n$. 
\end{proof} 

\subsection{Existence of solutions to the mean field game system} 

In order to construct a solution to the mean field game \eqref{mfg2}, we associate
to the energy in~\eqref{energia} a dual problem, using standard arguments in convex analysis. 
First of all, following \cite{bc}, we pass to a convex problem.

Given a minimizer $(\bar m, \bar w)$ as obtained in Theorem \ref{exthm}, 
we introduce  the following functional 
\begin{equation}\label{convex} 
J(m,w):= \int_{Q} m L\left(-\frac{w}{m}\right)+f(x,\bar m) m \,dx. 
\end{equation} 

We claim that for  $(m,w)\in\mathcal{K}$  we have that 
\[\int_{Q} m L\left(-\frac{w}{m}\right)dx-\int_{Q} \bar m L\left(-\frac{\bar w}{\bar m}\right)\geq -\int_Q f(x,\bar m) (m-\bar m) dx.\]
This can be proved as in \cite[Proposition 3.1]{bc},  using the convexity of $L$ and the regularity of $F$. 
The idea is to consider, for every $\lambda\in (0,1)$, 
$m_\lambda:=\lambda m +(1-\lambda)\bar m$ and the same definition 
for $w_\lambda$, and to observe that by minimality
\[\int_{Q} m_\lambda L\left(-\frac{w_\lambda}{m_\lambda}\right)dx-\int_{Q} 
\bar m L\left(-\frac{\bar w}{\bar m}\right)\geq -\int_Q F(x,\bar m_\lambda)-F(x,\bar m)  dx\,.\]
Then, using the
convexity to estimate the left hand side and the
regularity of $F$  on the right hand side, and finally sending $\lambda\to 0$,
we get that \[\min_{(m,w)\in\mathcal{K}} J(m,w)=J(\bar m, \bar w).\]
\smallskip

Now we complete the proof of Theorem~\ref{TH:MAIN}, by showing 
the last point~$(3)$.

\begin{theorem}\label{solmfg}
Let $(\bar m, \bar w)$ be a minimizer of $J$ as given by Theorem \ref{exthm}. 

Then, $\bar m\in H^{2s-1}_p(Q)$, for all~$p>1$,
and there exist~$\lambda\in \R$ and~$ u \in C^{2s+\alpha}(Q)$,
such that $( u, \lambda, \bar m)$
is a classical solution to the mean field game \eqref{mfg2}.
Finally  $\bar w=-\bar m \nabla H(\nabla   u)$. 
\end{theorem} 

\begin{proof} 
The functional in~\eqref{convex} is convex, so we can write the dual problem as follows, following standard arguments in convex analysis. 
Recall that $\bar m\in C^{\alpha}(Q)$,
for any $\alpha\in \left(0, 2s-1-\frac{N}{\gamma'}\right)$,
thanks to Proposition~\ref{pstime}.

Now, we consider the following functional
\[\mathcal{A}(m,w,u, c):= \int_{Q} \left[ m L\left(-\frac{w}{m}\right)+f(x,\bar m) m -m(-\Delta)^{s} u+\nabla u \cdot w-c m \right]dx +c . \] 
It is easy to observe that 
\begin{equation}\label{m1} J(\bar m,\bar w)=
\inf_{ \{ (m,w)\in (L^1(Q)\cap C^{\alpha}(Q))\times L^1(Q), 
\ \int_Q m=1, m\geq 0\}}\sup_{(u,c)\in C^{2s }(Q)\times \R} 
\mathcal{A}(m,w,u, c),\end{equation}
so the infimum is actually a minimum. 

Note that $\mathcal{A}(\cdot, \cdot, u,c)$ is convex and
weak$*$ lower semicontinuous and $\mathcal{A}(m,w, \cdot, \cdot)$ is linear (so in particular concave).  
So we can use the min-max Theorem, see~\cite{et},
and interchange minimum and supremum, that is 
\begin{equation}\begin{split}\label{m2}
&\min_{\{(m,w)\in L^1\cap C^{\alpha}\times L^1,
\ \int m=1, m\geq 0\} }\sup_{(u,c)\in C^{2s}\times \R}
\mathcal{A}(m,w,u,c)\\
=&\,\sup_{(u,c)\in C^{2s}\times \R} \min_{\{(m,w)\in
L^1\cap C^{\alpha}\times L^1, \ \int m=1, m\geq 0\} }
\mathcal{A}(m,w,u, c).\end{split}\end{equation}

Finally, thanks to the Rockafellar's Interchange Theorem \cite{r}  between infimum and integral 
(based on measurable selection arguments
and the lower semicontinuity of the functional)
we get, using the fact that $H$ is the Legendre transform of $L$, that
\begin{equation}\begin{split}\label{m3}  
&\min_{\{(m,w)\in L^1\cap C^{\alpha}\times L^1, \ \int m=1, m\geq 0\} }
\mathcal{A}(m,w,u, c) \\ 
=\,&\int_{Q}\min_{m\geq 0,w} \left[ m L\left(-\frac{w}{m}\right)+f(x,\bar m) m -m(-\Delta)^{s} u+\nabla u \cdot w-c  m \right]dx+c\\ \nonumber 
=\,&\int_Q \min_{m\geq 0} m\left[-H(\nabla u) -(-\Delta)^{s} u+f(x,\bar m)-c\right]dx+c.
\end{split}\end{equation} 

Note that \[ \min_{m\geq 0} m\left[ 
 -(-\Delta)^{s} u-H(\nabla u)+f(x,\bar m)-c\right]=
\begin{cases} 0, &{\mbox{ if }} -(-\Delta)^{s} u -H(\nabla u) +f(x,\bar m)-c\geq 0,\\
-\infty, & {\mbox{ if }} - (-\Delta)^{s} u-H(\nabla u)  +f(x,\bar m)-c< 0.\end{cases}\] 
Therefore, from \eqref{m1}, \eqref{m2} and~\eqref{m3} we get that 
\begin{equation}\begin{split}
\label{m4} J(\bar m,\bar w)=&\,\sup_{(u,c)
\in C^{2s}\times \R}\int_Q \min_{m\geq 0}
m\left[-H(\nabla u) - (-\Delta)^{s} u+f(x,\bar m)-c\right]dx+c\\
=&\,\sup\left\{c\in \R \ | \exists u\in C^{2s},  \text{ s.t. }
(-\Delta)^{s} u+H(\nabla u) +c \leq f(x,\bar m)\right\}.
\end{split}\end{equation}
 Due to Theorem \ref{ergodic}, such supremum is actually a maximum:
there exist $\lambda\in\R$ and a periodic function~$u\in C^{2s+\alpha}(Q)\cap H^{2s}_p(Q)$,
for every $\alpha< 2s-1-\frac{N}{\gamma'}$ and every $p>1$, 
which is unique up to additive constants and solves  
\begin{equation}\label{hj}
(-\Delta)^{s} u+H(\nabla u) +\lambda =f(x,\bar m).\end{equation}

So, equality \eqref{m1} reads  \[\lambda= J(\bar m,\bar w)=\int_{Q} \bar m \left[ L\left(-\frac{\bar w}{\bar m}\right)+f(x,\bar m)\right]dx.\] 

Therefore, recalling that $\int_Q\bar m=1$ and using both \eqref{hj} and  \eqref{eqm} 
with test function $u$, we obtain that
\begin{eqnarray*}
0&=&\int_{Q} \bar m \left[ L\left(-\frac{\bar w}{\bar m}\right)+f(x,\bar m)-\lambda\right]dx\\
&=&
\int_{Q} \bar m \left[ L\left(-\frac{\bar w}{\bar m}\right)+(-\Delta)^{s} u+H(\nabla u) 
\right]dx\\ &=&\int_{Q} \bar m \left[ L\left(-\frac{\bar w}{\bar m}\right)+ \nabla u\cdot 
\frac{\bar w}{\bar m}+H(\nabla u) \right]dx. \end{eqnarray*} 
Using the fact that $H$ is the Legendre transform of $L$ and \eqref{leg}, we thus conclude
that  \[\frac{\bar w}{\bar m}=-\nabla H(\nabla u),\]
where $\bar m\neq 0$. 
Moreover, by the definition of $L$, we get that $\bar w=0$ where $\bar m=0$. 

In particular, recalling \eqref{eqm}, we find that $\bar m$ is a solution  of 
\[ (-\Delta)^{s}m - \div( m \nabla H(\nabla u))=0, \qquad{\mbox{ with }}
\quad \int_Q  m=1.\]
Since $\bar m\nabla H(\nabla u)\in L^\infty(Q)$, by Lemmata~\ref{lemmauno}
and~\ref{lemmadue}, we get that $\bar m\in H^{2s-1}_p(Q)$ for every $p>1$. 
This implies that $(u,\bar m)$ is a solution to \eqref{mfg2}.
The proof of Theorem~\ref{solmfg} is thus complete. 
\end{proof} 

\section{Further properties: improved regularity and uniqueness}\label{sectionimprovement}

If we assume some more regularity on $f$ and $H$, we can obtain more regular solutions. 

\begin{proposition}
Assume that $H\in C^{1+k}(\R^N)$, for some $k\geq 1$, and,
in the local case (under assumptions \eqref{assFlocal1} or
\eqref{assFlocal}), that~$f\in C^k(\R^N\times \R)$
or, in the nonlocal case (under assumption \eqref{assFnonlocal}), that~$f$ maps continuously~$X$ (as defined in \eqref{Xdef})
in a bounded subset of $C^{k}(Q)$. 
  
Then, the system in~\eqref{mfg2} admits a classical solution $(u,\lambda, m)\in C^{k}(Q) \times 
\R\times C^{k-1}(Q)$.
\end{proposition} 

\begin{proof} By Theorems \ref{solmfgreg}, \ref{solmfgbdd} and \ref{solmfg}, we have a solution $(u,\lambda, m)\in
(C^{2s+\alpha}(Q)\cap H^{2s}_p(Q))\times \R\times H^{2s-1}_p(Q)$. 

Using the regularity of $m$ and $\nabla H(\nabla u)$ and the fact that both are in $L^\infty$, 
 we get that $m\nabla H(\nabla u))\in H^{2s-1}_p(Q)$ for all $p>1$. 
This implies by Lemma~\ref{lemmauno} that $m\in H^{4s-2}_p(Q)$ for every $p>1$.

By the regularity of $f$, if $k\geq 4s-2$,  also $f(m)\in H^{4s-2}_p(Q)$. 
Therefore, using the Hamilton-Jacobi equation, we get that 
$$(-\Delta)^s u\in H^{2s-1}_p(Q)$$
for every $p>1$, which gives that~$u \in H^{4s-1}_p(Q)$.

Reasoning as above we obtain that  $m\nabla H(\nabla u))\in H^{4s-2}_p(Q)$ for all $p>1$,
and then by Lemma~\ref{lemmauno} we conclude that $m\in  H^{6s-3}_p(Q)$ for every $p>1$. 

We iterate the argument up to arriving to $m\in C^{M(2s-1)}(Q) $,
where $M:=\left[\frac{k}{2s-1}\right]$ (that is,
$M$ is the integer part of $\frac{k}{2s-1}$). In particular $m\in C^{k-1}(Q)$. 
So we conclude that $u\in C^{2s+N(2s-1)}(Q)$.  
\end{proof} \

It is well known that, under a monotonicity condition on the function $f$,
we have uniqueness of solutions to mean field games system.

\begin{theorem}\label{un} Assume that  
 in the local case  (under assumptions \eqref{assFlocal1} or  \eqref{assFlocal}) the map~$m\mapsto  f(x,m)$  is increasing for 
all $x\in Q$ or in the nonlocal case   (under assumption \eqref{assFnonlocal}) 
\[\int_Q (f[m_1](x)-f[m_2](x))(m_1-m_2) dx >0, \qquad {\mbox{for any }} m_1,m_2 \in X.\] 

Then, the system in~\eqref{mfg2} admits a unique classical solution 
$(u, \lambda, m)$, where $u$ is defined up to addition of constants. 
\end{theorem}

\begin{proof} 
The argument is standard, see \cite{ll}, and the adaptation to the fractional case is straightforward. 
%Note that under the assumptions of Theorem~\ref{un}, the energy~\eqref{energia}
%is strictly convex. Suppose by contradiction that
%there exist two different solutions $(u_1, \lambda_1, m_1)$ and $(u_2, \lambda_2, m_2)$. 
%Recalling Proposition \ref{proconvex}, by the uniqueness of minimizers,
%we get that $m_1=m_2$. This implies, by Theorem \ref{ergodic}, that  $\lambda_1=\lambda_2$ and 
%that $u_1=u_2+C$ for some constant $C$. The proof of Theorem~\ref{un} is thus complete. 
\end{proof}

\end{document}